\documentclass[10pt,letterpaper]{article}
\usepackage[top=0.85in,left=1.75in,footskip=0.75in,marginparwidth=2in]{geometry}

\usepackage[utf8]{inputenc}

\usepackage{cite}

\usepackage{nameref,hyperref}

\usepackage{microtype}
\DisableLigatures[f]{encoding = *, family = * }

\raggedright
\usepackage{changepage}

\usepackage{verbatim}

\usepackage[aboveskip=1pt,labelfont=bf,labelsep=period,singlelinecheck=off]{caption}

\makeatletter
\renewcommand{\@biblabel}[1]{\quad#1.}
\makeatother

\usepackage{lastpage,fancyhdr,graphicx}
\usepackage{epstopdf}
\fancyheadoffset[L]{2.25in}
\fancyfootoffset[L]{2.25in}

\usepackage{color}

\definecolor{Gray}{gray}{.25}

\usepackage{graphicx}

\usepackage{sidecap}

\usepackage{wrapfig}
\usepackage[pscoord]{eso-pic}
\usepackage[fulladjust]{marginnote}
\reversemarginpar

\usepackage{amsthm}
\newtheorem{Pro}{Proposition}

\newtheorem{Def}{Definition}
\newtheorem{Lem}{Lemma}[section]
\newtheorem{Col}{Corollary}[section]
\newtheorem{Thm}{Theorem}[section]
\newtheorem{Rmk}{Remark}
\newtheorem{Emp}{Example}
\usepackage{amsmath}
\usepackage{MnSymbol} 
\usepackage{tikz-cd}
\usepackage{mathtools}
\usepackage{amsfonts}

\usepackage{chemformula}
\usepackage[autostyle]{csquotes}
\usepackage{extarrows}
\usepackage{graphicx}

\DeclareMathOperator{\vz}{\textbf{z}}
\usepackage{mathtools,xparse}

\DeclarePairedDelimiter{\norm}{\lVert}{\rVert}
\NewDocumentCommand{\normL}{ s O{} m }{%
  \IfBooleanTF{#1}{\norm*{#3}}{\norm[#2]{#3}}_{L_2(\Omega)}%
}

\begin{document}
\vspace*{0.35in}

\begin{flushleft}
{\Large
\textbf\newline{Relative Periodic Solutions Of The N-Vortex
    Problem Via The Variational Method}
}
\newline


Qun WANG\textsuperscript{*}\\
\bigskip
Universit\'e Paris-Dauphine, PSL Research University, CNRS, UMR 7534, CEREMADE,
F-75016 Paris, France\\
\bigskip
* wangqun927@gmail.com
\end{flushleft}

\section*{Abstract}
This article studies the N-vortex problem in the plane with positive vorticities. After an investigation of some properties for normalised relative equilibria of the system, we use symplectic capacity theory to show that, there exist infinitely many normalised relative periodic orbits on a dense subset of all energy levels, which are neither fixed points nor relative equilibria.

\section{Introduction}
\label{sec:1}
\subsection{The N-Vortex Problem in the Plane}
\label{subsec:1.1}
The study of vortex dynamics dates back to Helmholtz's work on
hydrodynamics in 1858 \cite{helmholtz}. It has been linked to
superfluids, superconductivity, and stellar
system\cite{lugt1983vortex}. Known as the Kirchhoff Problem, its
Hamiltonian structure is first explicitly found
by Kirchhoff \cite{kirchhoff1876vorlesungen} for $\mathbb{R}^2$, and later on
generalized by Routh \cite{routh1880some} and then Lim
\cite{lim1943motion} to general domains in the plane. Here
we consider the problem in the plane,
\begin{equation}
\mathbf{\Gamma} \dot{\vz} (t) = \textit{X}_H(\vz(t)) = \mathcal{J}_N\nabla
  H(\vz(t)), \quad \dot{\vz}=(z_1,z_2,...,z_N),\quad  z_i=(x_i,y_i)\in
  \mathbb{R}^2 \label{sys:H1}  \tag{H1}
\end{equation}
where the Hamiltonian is 
\begin{equation}
  \label{eq:Ham}%
  H(z)= -\frac{1}{4\pi}\sum_{1 \leq i< j\leq
    N}\Gamma_{i}\Gamma_{j}\log{|z_i -z_j|^2}
\end{equation}
while the Poisson matrix $\mathcal{J}_N$ and the vorticity matrix $\Gamma$ are
\begin{align}
\mathcal{J}_{N} = 
\begin{bmatrix}
\mathbb{J} &       &      \\  
          & \ddots &      \\
          &       & \mathbb{J}
\end{bmatrix}, \quad 
\mathbb{J} =
\begin{bmatrix}
0  & 1 \\
-1 & 0
\end{bmatrix}\\
\mathbf{\Gamma} = 
  \begin{bmatrix}
    \Gamma_1 &&&&&\\
    & \Gamma_1 &&&&\\
    &&\ddots &&&  \\
    &&&\ddots &&   \\ 
    &&&&\Gamma_N & \\
    &&&&& \Gamma_N 
  \end{bmatrix}.
\end{align}
It is understood that 
\begin{itemize}
\item $N$ is the number of vortices;
\item $z_i = (x_i, y_i)$ is the position of the i-th vortex in the plane;
\item $\Gamma_i \in \mathbb{R}\setminus\{0\}$ is the vorticity of the i-th vortex;
\item $\textbf{z}= (z_1,z_2,...,z_N)$ is the vortices configuration;
\item $\textit{X}_H$ is the Hamiltonian vector field of $H$;
\item $\mathcal{J}_N$ is a $2N\times 2N$ block-diagonal matrix;
  obtained by putting $N$ copies of $2\times 2$ matrice $\mathbb{J}$
  on the diagonal.
\end{itemize}
Sometimes we will need to consider a sequence of vortices
configurations. In that case we will denote this sequence
$\{\textbf{z}^k\}_{k\in \mathbb{N}}$, with upper indices as opposed to
lower indices refering to particular vorticies. The quantity
\begin{align}
 L = \sum_{1\leq i< j\leq N} \Gamma_i\Gamma_j
\end{align}
is called the \textbf{total angular momentum} and will be used
frequently; if $V$ is a subset of \{1,2,...,N\}, we also define
$\displaystyle L_V = \sum_{i,j \in V, i< j}\Gamma_i\Gamma_j$.
Finally, throughout this article, if not explicitly emphasized, \textbf{we
always suppose that all vortices are of positive vorticity} :
\begin{equation}
  \label{Hyp:positive}
  \Gamma_i > 0 \quad (1\leq i \leq N).
  \tag{Hypo}
\end{equation}
Hence $L>0$, and
$L_V>0$ for all $V$'s.
\subsection{Symmetries, First Integrals, and Integrability}
\label{subsec:1.2}
Let $\mathbb{E}(2)= \mathbb{T}(2)\rtimes \mathbb{O}(2)$ be the Euclidean group, where $\mathbb{O}(2)$ is the orthogonal group and $\mathbb{T}(2)$ is the translation group. Consider the action 
\begin{align*}
&\mathbb{E}(2) \times (\mathbb{R}^{2})^N \rightarrow (\mathbb{R}^{2})^N\\
&(g,\textbf{z}) \rightarrow  g\textbf{z}  
\end{align*}
where
\begin{align*}
g\textbf{z} = (gz_1,gz_2,...,g z_N), \forall g\in \mathbb{E}(2), z= (z_1,...,z_N) \in \mathbb{R}^{2N}
\end{align*}
Note that system \eqref{sys:H1} is invariant under both translation and
rotation, thus the corresponding quantities
\begin{align}
P(\textbf{z}) = \sum_{1\leq i\leq N}\Gamma_i x_i,\quad Q(\textbf{z})= \sum_{1\leq i \leq N} \Gamma_i y_i,\quad I(\textbf{z}) = \sum_{1\leq i\leq N} \Gamma_i|z_i|^2
\end{align}
are first integrals. These first integrals are \textbf{not} in
involution in general with respect to the Poisson bracket
\begin{align*}
\{f,g\} = \sum_{1\leq i\leq N} \frac{1}{\Gamma_i} (\frac{\partial f}{\partial y_i}\frac{\partial g}{\partial x_i}-\frac{\partial f}{\partial x_i}\frac{\partial g}{\partial y_i}).
\end{align*}
Actually, 
\begin{align}
\{P,I\} = -2Q,\quad \{Q,I\} = 2P.
\end{align}
On the other hand, note that $H, I, P^2+ Q^2$ are independent first
integrals in involution. Hence the 3-vortex problem is integrable. For
$N\geq 4$, the $N$-vortex problem is in general not integrable
\cite{ziglin1980nonintegrability,koiller1989non}. This allows us to
draw a parallel between the $N$-vortex problem with $N \geq 4$ and the
$N$-body problem with $N \geq 3$, and thus side with Poincar{\'e} when he
famously described the study of periodic orbits as \enquote{the only
  opening through which we can try to penetrate in a place which, up
  to now, was supposed to be
  inaccessible}\cite[section~36]{poincare1892methodes}. In this article, we
study the existence of relative periodic orbits in the $N$-vortex
system, i.e., orbits which are periodic up to rotations. Note that from
experimental point of view, it is relative equilibria that have been
first realized for vortices of superfluid $\ch{^4He}$
\cite{yarmchuk1979observation}. Hence from either theoretical or
practical consideration, relative periodic orbit will be an ideal
candidate for our analysis.

\subsection{Normalised Orbits}
\label{subsec:1.3}
The closed orbits of N-vortex problem \eqref{sys:H1} are not
isolated. Indeed, if $\textbf{z}(t)$ is an orbit, then so are
\begin{itemize}
\item $(z_1(t)+c, \cdots, z_N(t)+c)$, $c\in \mathbb{R}^2$
\item $\lambda^{\frac{1}{2}}\textbf{z}(\frac{t}{\lambda})$, $
  \lambda \in \mathbb{R}^{+}$. 
\end{itemize}
We wish not to distinguish such orbits. To this end, we give the following definition:

\begin{Def}
\label{Def:orbit}
we will call an orbit $\textbf{z}(t)$ of the system \eqref{sys:H1}
\begin{enumerate}
\item \textbf{centred} if it satisfies $P(\textbf{z}(t)) = Q(\textbf{z}(t)) = 0$
\item \textbf{normalised} if it is centred and satisfies $I(\textbf{z}(t)) = 1$
\item \textbf{periodic} if $\textbf{z}(t) = \textbf{z}(t+T)$ for some $T>0$
\item \textbf{relatively periodic orbit (RPO)} if $\textbf{z}(t) = g \textbf{z}(t+T)$ for some $T>0$ and $g \in \mathbb{E}(2)$
\end{enumerate}
\end{Def}
Thus, the abbreviation \textbf{NRPO} will stand for a \textbf{normalised relative periodic
  orbit}. Note that in particular for a NRPO we have $g\in \mathbb{O}(2)$ in the above definition.
A periodic solution of the planar $N$-vortex problem s.t. $\sum_{i=1}^N\Gamma_{i} \neq 0$ is called a
\textbf{relative equilibrium}, if it is of the form
\begin{align*}
z_i(t)= e^{\mathbb{J}\omega t}(z_i(0)-c)+ c
\end{align*}
where $\displaystyle c = \frac{\sum_{i=1}^N \Gamma_i z_i}{\sum_{i=1}^N \Gamma_i }$ is the vorticity center. This is a special configuration where all the vortices rigidly rotate about their center of vorticity $c$. In particular, given a normalised relative equilibrium, i.e., $c = (0,0)$ and $I(\vz) =1$, it is of course a NRPO. We define
\begin{align*}
&\mathcal{Z}_{0}(H)= \{\textbf{z}| \text{\textbf{z} is a normalised orbit of the system } \eqref{sys:H1} \}\\
&\mathcal{Z}_{1}(H) =\{\textbf{z}| \text{\textbf{z} is a normalised relative equilibrium of the system } \eqref{sys:H1} \}
\end{align*}
We list some properties that will be used frequently later on:
\begin{Pro}
\label{Pro:NRE}
The following are equivalent:
\begin{align}
(1)& \quad \textbf{z} \in \mathcal{Z}_{1};      \label{E1} \\
(2)&  \quad \nabla H(\textbf{z}) = -\frac{L}{4\pi} \nabla I(\textbf{z}) \label{E2}
\end{align}
\end{Pro}
\begin{proof}:
(1)$\Rightarrow$ (2) : By definition of relative equilibrium, $\textbf{z}(t)\in \mathcal{Z}_{1}$ implies $\exists \omega\in \mathbb{R}$ s.t. 
\begin{align*}
\nabla H(\textbf{z}(t)) = \frac{\omega}{2} \nabla I(\textbf{z}(t)) 
\end{align*}
taking inner product with $\textbf{z}(t)$ on both sides. Since $I(\textbf{z})=1$, one sees that 
\begin{align*}
-\frac{L}{2\pi} = \omega I(\textbf{z}(t)) \Rightarrow \frac{\omega}{2} = -\frac{L}{4\pi}
\end{align*}
Hence (2) is proved.\\
(2)$\Rightarrow$ (1) : If $\textbf{z}$ satisfies that $\displaystyle \nabla H(\textbf{z}) = -\frac{L}{4\pi} \nabla I(\textbf{z})$, then the flow passing through $\textbf{z}$ will be a relative equilibrium. We need to show that such a relative equilibrium is normalised. First, by considering $(x,y)\in \mathbb{R}^2 $ as a complex number $x+iy \in \mathbb{C}$, (\ref{E2}) implies that 
\begin{align*}
-\frac{1}{2\pi}\sum_{j\neq i } \Gamma_i\Gamma_j \frac{\bar{z}_i -\bar{z}_j }{|z_i-z_j|^2} = -\frac{L}{4\pi} \Gamma_i \bar{z}_i, \quad \forall 1\leq i \leq N
\end{align*}
It follows that 
\begin{align*}
0 = -\frac{1}{2\pi}\sum_{i=1}^{N}\sum_{j\neq i } \Gamma_j\Gamma_i \frac{\bar{z}_i -\bar{z}_j }{|z_i-z_j|^2} = -\sum_{i=1}^{N}\frac{L}{4\pi} \Gamma_i \bar{z}_i  
\end{align*}
Thus $\sum_{i=1}^{N}\Gamma_i z_i =0 $, and $\textbf{z}$ is centred. Next, multiply $\textbf{z}$ on both sides of (\ref{E2}), so that $ \displaystyle-\frac{L}{2\pi} = \nabla H(\textbf{z}) \textbf{z} =-\frac{L}{4\pi} \nabla I(\textbf{z})\vz = -\frac{L}{2\pi} I(\textbf{z})$. Thus $I(\textbf{z})=1$.  
\end{proof}
The first such configuration, found In 1883 by J.J. Thomson, is the
so-called Thomson configuration \cite{thomson1883treatise}, i.e., $N$
identical vortices located at the vertices of a N-polygon and rotating
uniformly around its center of vorticity (which could be fixed to the
origin).

\begin{figure}
\begin{center}
\includegraphics[width=40mm,scale=0.5]{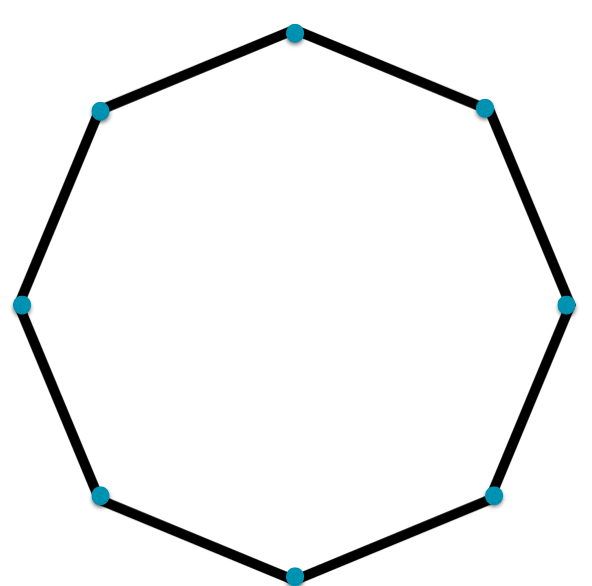}
\captionsetup{justification=centering}
\caption{Thomson configuration for 8 vortices which form an octagon}
\label{fig:1}
\end{center}
\end{figure}
There has since then been extensive studies on relative equilibria in
the planar $N$-vortex problem, see for example
\cite{palmore1982relative,o1987stationary,roberts2013stability}. The
study of relative equilibria is a subject in itself. Although their
number and even their finiteness are unknown as functions of $N$ (see  \cite{hampton2009finiteness,o2007relative} for the special case when $N = 4$), one
does not expect that relative equilibria could in general be abundant in the
phase space. Hence our interest will be on the RPOs that are not relative
equilibria.

\begin{Def}
\label{def:NTNRPO}
We say an orbit \textbf{z}(t) is a \textbf{non-trivial normalised relative periodic orbit} (NTNRPO) if \textbf{z}(t) is a normalised relative periodic orbit but \textbf{not} a normalised relative equilibrium.
\end{Def}

NTNRPO's could conjecturally be dense in some open sets of the phase
space, similarly to what Poincar{\'e} conjectured for the $N$-body
problem. Correspondingly define
\begin{align*}
  \mathcal{Z}_{2}(H)= \{\textbf{z}| \text{\textbf{z} is a 
  non-trivial normalised relative periodic orbit of the system } \eqref{sys:H1} \}.
\end{align*}
So, in this article, we focus on the search of NTNRPOs of the
$N$-vortex problem. In 1949, Synge has given a thorough study of
relative periodic solutions of three vortices in
\cite{synge1949motion} (see also \cite{aref1979motion}). Since the
$N$-vortex problem is in general not integrable, it is thus more
complicated to find NTNRPOs when more vortices are presented. Several
difficulties occur in the search of periodic solutions for system \eqref{sys:H1}, for example,
\begin{itemize}
\item the system is singular around collisions/infinity
\item energy surfaces are not compact
\item the Hamiltonian is not convex
\end{itemize}
Due to these difficulties, methods in \cite{ekeland2012convexity,hofer2012symplectic,carminati2006fixed} cannot be applied
directly. Some NTNRPOs can be found by applying various perturbative
arguments around relative equilibria, as in
\cite{borisov2004absolute}; see also \cite{carvalho2014lyapunov} for
the application of the Lyapunov centre theorem, and
\cite{bartsch2016periodic,bartsch2017global} for the application of
degree theory. We would like to make some contribution for the general
knowledge on existence of such orbits whose energy might be far from
the energy of the relative equilibra with an arbitrary number $N$ of
vortices and for arbitrary positive vorticity $\mathbf{\Gamma}$. To this end
some global method is needed. We focus on variational arguments,
instead of perturbative methods. The existence of non-relative
equilibrium solution will rely on the absence of relative equilibrium
(see \textbf{Figure} \ref{fig:2}). Define
\begin{align*}
\mathcal{H}_{0}& = \{h \in \mathbb{R} | h = H(\textbf{z}), \textbf{z}\in \mathcal{Z}_{0}(H)\}\\
\mathcal{H}_{1}& = \{h \in \mathbb{R} | h = H(\textbf{z}), \textbf{z}\in \mathcal{Z}_{1}(H)\}\\
\mathcal{H}_{2}& = \{h \in \mathbb{R} | h = H(\textbf{z}), \textbf{z}\in \mathcal{Z}_{2}(H)\}\\
\end{align*}
Note that these are well-defined because the Hamiltonian $H$ is
autonomous, thus it is constant along its flow. Clearly
$\mathcal{H}_{2} \subset \mathcal{H}_{0} $. Our main
result is that:
\begin{Thm}
\label{Thm:Main}
Under hypothesis \eqref{Hyp:positive}, $\mathcal{H}_{2}$ is dense in $\mathcal{H}_{0}$.
\end{Thm}
\begin{figure}
\label{VariPert}
\captionsetup{justification=centering}
\begin{center}
\includegraphics[width=70mm,scale=0.5]{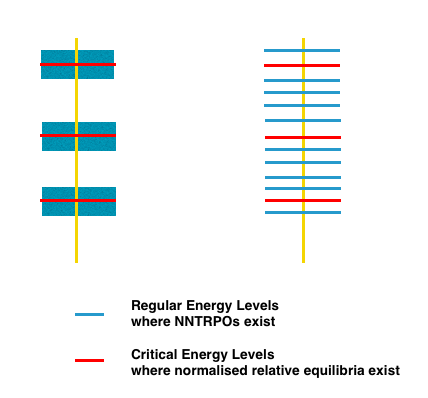}
\caption{Perturbative Method(Left) VS Variational Method(Right).}
\label{fig:2}
\end{center}
\end{figure}
The rest of the article is organized in the following way:
\begin{itemize}
\item In chapter \ref{sec:2}, we study the normalised relative equilibria of $H$. In particular, we show that the they are isolated from the region of singularity (\textbf{Lemma \ref{Lem:MinDist}}). Using this fact we show further that the set of critical value is very small(\textbf{Theorem \ref{Thm:ClosedNull}}). In case of positive rational vorticity, it even can be shown to be a finite set (\textbf{Theorem \ref{Thm:RationalFinite}});
\item In chpater \ref{sec:3}, we show that by applying some symplectic reduction, we can focus on dynamics in the reduced phase space, where the energy level is compact (\textbf{Lemma \ref{Lem:compact}}). The capacity theory will then show that there are infinitely many NTNRPOs (\textbf{Theorem \ref{Thm:ManyRPOLocal}}). Combining results in chapter \ref{sec:2} and in chapter \ref{sec:3}, the main result(\textbf{Theorem \ref{Thm:Main}}) is thus proved;
\item In chapter \ref{sec:4}, we add discrete symmetry constraints to get NTNRPOs of special symmetric configuration.
\end{itemize}

We have resumed necessary technical background and some details of proofs in the appendix. 
\section{Normalised Relative Equilibria of H}
\label{sec:2}
\subsection{General Positive Vorticities}
\label{subsec:2.1}
Before we proceed to study NTNRPOs, we first need to have some preparation for properties of the normalised relative equilibria of $H$. In this chapter, we study the normalised relative equilibria of H, with an emphasis on their energy levels. First note that the mutual distances between vortices in a normalised relative equilibrium configuration cannot be too small. More precisely: 
\begin{Lem}
\label{Lem:MinDist}
For $\Gamma_i \in \mathbb{R}^{+}$, there exists constant $\epsilon(\mathbf{\Gamma})$ which depends only on the vorticities $\mathbf{\Gamma} = (\Gamma_1,\Gamma_2,..\Gamma_N), 1\leq i\leq N$, s.t.
\begin{align*}
\inf_{\substack{\textbf{z}\in \mathcal{Z}_{1}\\ 1\leq i < j\leq N}} |z_i-z_j|^2 >\epsilon > 0
\end{align*}
\end{Lem}
\begin{Rmk}
As the relative equilibria are rigid body motions, we have dropped the dependence of time of $\textbf{z}$ to simplify the discussion. 
\end{Rmk}
This result first appears in the work of O'Neil \cite{o1987stationary} and has been reproved recently by Roberts \cite{roberts2017morse} using a renormalisation argument, followed by a detailed discussion on Morse index of relative equilibria. We here give an alternative proof by the  observation that for a relative equilibirum, the vorticity center of a given cluster also rotates uniformly. 
\begin{proof}:
Denote
\begin{align*}
m(\textbf{z}) = \inf_{ 1\leq i < j\leq N} |z_i-z_j|^2
\end{align*} 
Suppose to the contrary that $\textbf{z}^k$ is a sequence of relative equilibria whose mutual distances s.t. 
$\lim_{k\rightarrow \infty}m(\textbf{z}^k) = 0$. Then by consecutively passing to subsequence if necessary, we may suppose that there exists an sub-index set $V\subset\{1,2,..,N\}$ s.t. $z^k_i \rightarrow z^*, \forall i\in V$.  Denote $\textbf{z}_V$ as the vector of vortices with index in V. The Hamiltonian could be separated into two parts, the interactions between vortices in V and otherwise. Let $H(\textbf{z}) = H_V(\textbf{z}) + H_{V^c}(\textbf{z})$, where
\begin{align}
H_{V}(\textbf{z})&= -\frac{1}{4\pi}\sum_{\substack{i< j \\ i,j\in V}}\Gamma_{i}\Gamma_{j}\log{|z_i -z_j|^2} \\
H_{V^c}(\textbf{z})&=-\frac{1}{4\pi}\sum_{\substack{i< j \\ (i,j)\notin V\times V}}\Gamma_{i}\Gamma_{j}\log{|z_i -z_j|^2}
\end{align}
It follows that 
$\displaystyle \nabla H(\textbf{z}^k) \textbf{z}^k = -\frac{1}{2\pi}L, \text{while } \nabla H_V(\textbf{z}_V^k) \textbf{z}_V^k = -\frac{1}{2\pi}L_V$. 
Observe that $c^k_V$, the vorticity centre of $\textbf{z}^k_V$, also follows a uniform rotation with the vortices. As a result, 
\begin{align}
\dot{c}^k_V &= \frac{\sum_{i\in V}\Gamma_i\dot{z}^k_i}{\sum_{i\in V} \Gamma_i} =\mathbb{J} \frac{\omega}{2} c^k_V\\
\Gamma_i\dot{z}^k_i &= \mathbb{J}(\nabla_{z_i} H_V(\textbf{z}) +  \nabla_{z_i} H_{V^c}(\textbf{z}) ) = \mathbb{J} \Gamma_i \frac{\omega}{2} z_i^k,\quad i\in V
\end{align}
Since $\lim_{k\rightarrow \infty}c^k_V = \lim_{k\rightarrow \infty} z^k_i  = z^*,\forall i\in V $,
We see that $\displaystyle \lim_{k\rightarrow \infty}\nabla H_V(\textbf{z}_V^k) = \textbf{0}$. But we know already that $\displaystyle \nabla H_V(\textbf{z}_V^k) \textbf{z}_V^k = -\frac{1}{2\pi}L_V$. As $|z^i_V|$ is bounded (since $\vz^k \in \mathcal{Z}_1(H)$), this implies that $L_V =0$, which contradicts the fact that $\Gamma_i > 0,\forall i\in V$. As a result, such sequence $\textbf{z}^k$ does not exist. The lemma is proved.
\end{proof}
Lemma \ref{Lem:MinDist} tells us that the relative equilibria are isolated from the diagonals, where collision happens and singularity rises. With this result in hand, we will study the distribution of energy levels on which normalised relative equilibria exist. For a subeset $\mathcal{A}\subset \mathbb{R}$, we denote by $\mu(\mathcal{A})$ its Lebesgue measure. Roughly speaking, we show that $\mathcal{H}_{1}$ is somehow a small subset of $\mathbb{R}$. 
\begin{Thm}
\label{Thm:ClosedNull}
For $\Gamma_i \in \mathbb{R}^{+},\forall 1\leq i \leq N$, $\mathcal{H}_{1}$ is a closed set in $\mathbb{R}$. Moreover $\mu(\mathcal{H}_{1}) = 0$.
\end{Thm}
\begin{proof}:
Suppose given a sequence of real numbers $h^k \in \mathcal{H}_{1}$ s.t. $\lim_{k\rightarrow \infty } h^k \rightarrow h^{*}\in \mathbb{R}$. Then   by definition of $\mathcal{H}_{1}$, there exists a sequence of normalised relative equilibria $\textbf{z}^k \in \mathcal{Z}_{1}$ s.t. 
\begin{align}
H(\textbf{z}^k) = h^k \rightarrow h^{*}
\end{align}
Since $I(\textbf{z}^k) = 1$, $\textbf{z}^k \in \mathbb{R}^{2N}$ is a bounded sequence, hence $\displaystyle \textbf{z}^k \xrightarrow{k\rightarrow \infty} \textbf{z}^{*}$. Thanks to lemma \ref{Lem:MinDist}, we see that points in $\mathcal{Z}_{1}$ are isolated from collision, hence $H$ is smooth at these points. As a result
\begin{align}
\nabla H(\textbf{z}^{*})&= \lim_{k\rightarrow\infty} \nabla H(\textbf{z}^k) = \lim_{k\rightarrow \infty}  -\frac{L}{4\pi} \nabla I(\textbf{z}^k(t)) = -\frac{L}{4\pi} \nabla I(\textbf{z}^{*}) \\
I(\textbf{z}^*) &= \lim_{k\rightarrow \infty} I(\textbf{z}^k) = 1,\\ H(\textbf{z}^*) &= \lim_{k\rightarrow \infty} H(\textbf{z}^k) = \lim_{k\rightarrow\infty} h^k = h^{*}\end{align}
In other words, $z^{*} \in \mathcal{Z}_{1}$ and $H(z^*) = h^{*}$. Hence $\mathcal{H}_{1}$ is a closed set.\\
Next, consider the function 
\begin{align*}
&f: \mathbb{R}^{2N} \rightarrow \mathbb{R}\\
&f(\textbf{z}) = 2H(\textbf{z}) + \frac{L}{2\pi} I(\textbf{z})
\end{align*}
Now by proposition \ref{Pro:NRE} $\nabla f(\textbf{z}) = 0$ implies that $\textbf{z}\in \mathcal{Z}_{1}$, which is isolated from collision. Hence Sard's theorem applies and $f(\mathcal{Z}_{1})$ is a null set. But on $\mathcal{Z}_{1}$, one has $I(\textbf{z}) = 1$, hence $\mathcal{H}_1 = H(\mathcal{Z}_{1})$ is a null set too. The theorem is thus proved.
\end{proof}
One important consequence of theorem 2 is the following corollary:
\begin{Col}
\label{Col:OpenDense}
$\mathcal{H}_0  \setminus	 \mathcal{H}_1$ is an open dense subset of $\mathcal{H}_0$.
\end{Col}
\begin{proof}:
Immediately from theorem \ref{Thm:ClosedNull}.
\end{proof}
\subsection{Rational Positive Vorticities And Beyond}
\label{subsec:2.2}
So far corollary \ref{Col:OpenDense} is sufficient for our further need. But when vorticities are positive rational numbers we can do even more. Actually, if $\Gamma_i \in \mathbb{Q}^+$, we can even prove that there are only finitely many energy levels on which a normalised equilibrium exists. The proof of theorem \ref{Thm:RationalFinite} below depends on a transformation of Hamiltonian and some elimination theory in algebraic geometry. We have resumed the detailed proof in Appendix \ref{appendix:A}. 
\begin{Thm}
\label{Thm:RationalFinite}
If $\text{ }\Gamma_i \in \mathbb{Q}^+, 1\leq i \leq N$, then $\mathcal{H}_{1}$ is a finite set. 
\end{Thm}
\begin{proof}:
See appendix \ref{appendix:A}.
\end{proof}
Theorem \ref{Thm:RationalFinite} is interesting in its own right, although we still do not know whether the number of normalised relative equilibria configurations are finite or not. Actually, from the proof in appendix \ref{appendix:A}, we see that $\Gamma_i \in \mathbb{Q}^+$ is sufficient but not necessary. More generally, if $\displaystyle \frac{\Gamma_i}{\Gamma_j} \in \mathbb{Q}^{+},\forall 1\leq i< j\leq N$, the result will hold. In particular, this is case for identical vorticities:\\
\begin{Col}
\label{Col:IdenticalFinite}
If $\text{ }\Gamma_i = c\in \mathbb{R} \setminus \{0\}, 1\leq i \leq N$, then $\mathcal{H}_{1}$ is a finite set. 
\end{Col}
\section{Symplectic Reduction and Relative Periodic Orbits in the Plane}
\label{sec:3}
In this chapter, we will use standard symplectic reduction to study the Hamiltonian in a reduced phase space. In the first section, we give some properties for the generalized Jacobi variable introduced by Lim \cite{lim1989canonical}. The main result is the compactness of energy surface of the reduced Hamiltonian in the reduced phase space. We do not give explicit calculation for coordinates transformations in this chpater. Instead, a detailed example of the 5-vortex problem is studied with explicit coordinate transformation in Appendix \ref{CT}.
\subsection{Lim's generalized Jacobi coordinates}
\label{subsec:3.1}
We would like to fix the center of vorticity to the origin thus study only centred orbits. The reason is that, any non-centred relative equilibrium, when putting into a rotationing framework around the origin, might automatically become a relative periodic solution that is not a relative equilibrium. This situation is illustrated in figure \ref{fig:3}.
\begin{figure}
\label{RAnon}
\captionsetup{justification=centering}
\begin{center}
\includegraphics[width=60mm,scale=0.5]{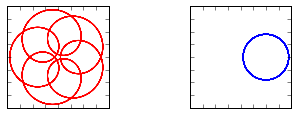}
\caption{A non trivial relative periodic (left) coming from a non-centred relative equilibrium in the original phase space (right)}
\label{fig:3}
\end{center}
\end{figure}
However, this kind of solution (orbits in red color in the left of figure \ref{fig:3}) is not the solution that we are searching for. Because it does not give any further insights about our dynamic system. As a result, we should insist on centred orbits, and we need some transformation to fix the vorticity centre to the origin.\\
The usual tool in celestial mechanics is the so called Jacobi coordinates. However, the usual Jacobi coordinates are not suitable for the $N$-vortex problems. This is because the conjugate variables $(q,p)$ are separated in the Hamiltonian for Newtonian gravitation N-body problem, i.e., 
\begin{align*}
H(q,p) = \frac{|p|^2}{2} + U(q) \tag{N-Body}
\end{align*}
while in $N$-vortex problem the conjudate variables $(x,y)$ are mixed
\begin{align*}
H(x,y)= -\frac{1}{4\pi}\sum_{i,j =1, i< j}^{N}\Gamma_{i}\Gamma_{j}\log{|z_i -z_j|^2} \tag{N-Vortex}\end{align*}
Hence if we perform a normal Jacobi transforamtion, we can fix the center of vorticity, but the resulting new Hamiltonian might be no longer invariant under rotation. There has been some study on symplectic transformations adapted to the $N$-vortex problem. For example \cite{khanin1982quasi,borisov1998dynamics,lim1989canonical} and so on. In particular, Lim's method in \cite{lim1989canonical} has introduced a canonical transformation for the $N$-vortex Hamiltonian based on graph theory. This transformation works particularly well when all the vorticities are postive, and is quite ideal for our purpose of evaluating the energy surfaces. We hence apply Lim's generalized jacobi coordinates to simplify our $N$-vortex system.\\
First, we make the change of variable 
\begin{align}
Z_i = (X_i, Y_i) = (\sqrt{\Gamma_i}x_i,\sqrt{\Gamma_i}y_i )
\end{align}
It turns out that $\textbf{Z}=(Z_1,Z_2,...,Z_N)$ follows the usual Hamiltonian system 
\begin{equation}
\dot{\textbf{Z}}(t) = X_{\hat{H}}(\textbf{Z}(t)) = \mathcal{J}_{N}\nabla \hat{H}(\textbf{Z}(t)) \quad \textbf{Z}=(Z_1,Z_2,...,Z_N),\quad  Z_i\in \mathbb{R}^{2} \tag{H2} \label{sys:H2}
\end{equation}
where 
\begin{align*}
\hat{H}(\textbf{Z}) = -\frac{1}{4\pi}\sum_{i,j =1, i< j}^{N}\Gamma_{i}\Gamma_{j}\log{|\frac{Z_i}{\sqrt{\Gamma_i}} -\frac{Z_j}{{\sqrt{\Gamma_j}}}|^2} 
\end{align*}
Then for the new variables,
\begin{align*}
\hat{P}(\textbf{Z}(t)) = \sum_{1\leq i\leq N}\sqrt{\Gamma_i} X_i(t),\quad \hat{Q}(\textbf{Z}(t))= \sum_{1\leq i \leq N} \sqrt{\Gamma_i} Y_i(t),\quad \hat{I}(\textbf{Z}(t)) = \sum_{1\leq i\leq N} |Z_i(t)|^2
\end{align*}
are first integrals.
We identify till the end of this section the coordinate in $Z_k = (X_k, Y_k) \in \mathbb{R}^2$ to the complex number $Z_k = X_k + iY_k $. A transformation from $\mathbb{C}^N$ to $\mathbb{C}^N$ will also be considered as a transformation from $\mathbb{R}^{2N}$ to $\mathbb{R}^{2N}$.
\begin{Pro}
\label{Pro:LimTransform}
(\cite[page~263]{lim1989canonical}) There exists a linear transformation for the positive planar N-vortex problem 
\begin{align*}
\phi: \quad \mathbb{C}^N &\rightarrow \mathbb{C}^N\\
Z= (X,Y) &\xrightarrow{\phi} W=(q,p) 
\end{align*}
s.t. 
\begin{enumerate}
\item $\phi$ is unitary; 
\item In the new coordinate W = (q,p), one has \\
\begin{align}
\begin{cases}
\displaystyle q_N = \frac{\sum_{1\leq N}\sqrt{\Gamma_i}X_i}{\sum_{1\leq i\leq N} \Gamma_i}\\
\displaystyle p_N = \frac{\sum_{1\leq N}\sqrt{\Gamma_i}Y_i}{\sum_{1\leq i\leq N} \Gamma_i}
\end{cases}.
\end{align}
\end{enumerate}
\end{Pro}

Since
$\mathbb{U}(N) = \mathbb{O}(2N) \bigcap \mathbb{SP}(2N)$, the transformation
$\phi$, seen as a transformation $\mathbb{R}^{2N} \xrightarrow{\phi} \mathbb{R}^{2N}$, is thus a real linear symplectic transformation.
As a result, we see that $q_N$ is a first integral and $p_N$ as its conjugate variable is cyclic. We can thus fix $q_N = p_N = 0$, and get a reduced Hamiltonian on $\mathbb{R}^{2N-2}$: 
\begin{align}
\bar{H}(q_1,p_1, q_2, p_2,..., q_{N-1}, p_{N-1} ; q_N = p_N = 0) = \bar{H}(\textbf{W};W_N=0)
\end{align}
\\
Consider the dynamic system\\
\begin{equation}
\dot{\textbf{W}}(t) = \textit{X}_{\bar{H}}(\textbf{W}(t)) \tag{H3}  \label{sys:H3}
\end{equation}

We resume some properties of the new Hamiltonian $\bar{H}$: 
\begin{Pro}
\label{Pro:AfterLimTransform}
Consider the Hamiltonian system \eqref{sys:H3} and the original Hamiltonian system \eqref{sys:H1} and \eqref{sys:H2} . Then:
\begin{enumerate}
\item Any orbit of $\bar{H}$ is a centred orbit of $H$;\\
\item The system \eqref{sys:H3} is invariant under rotation;
\item Define 
\begin{align}
\bar{I}(W) = \sum_{1\leq i\leq N-1} (p_i^2+ q_i^2)
\end{align}
Then $\bar{I}(\textbf{W}) = \hat{I}(\textbf{Z})$.
\end{enumerate}
\end{Pro}
\begin{proof}:
These propositions are direct consequences of the special symplectic transformation $\phi$.\\
1. $(q_N,p_N)$ corresponds to the vorticity centre in the original Hamiltonian and they are fixed at 0. Hence all the orbits of $\bar{H}$ are centred orbit of $H$.\\
2. $\phi$ is a linear transformation $\mathbb{C}^N \xrightarrow{\phi} \mathbb{C}^N$. The term $\displaystyle \log|\frac{Z_i}{\Gamma_i}-\frac{Z_j}{\Gamma_j}|^2$
under the transformation $\phi$ now becomes 
\begin{equation}
\displaystyle \log|\frac{Z_i}{\sqrt{\Gamma_i}}-\frac{Z_j}{\sqrt{\Gamma_j}}|^2= \log|\frac{\sum_{1\leq k\leq N-1} c_{ki} W_i}{\sqrt{\Gamma_i}}-\frac{\sum_{1\leq k\leq N-1} c_{kj} W_j}{\sqrt{\Gamma_j}}|^2 \label{eq:quatratic}
\end{equation}
where the coefficients $c_{ki}$ and $c_{kj}$ are decided by $\phi$. It is clearly still invariant under rotation. \\
3. We know that $I(\textbf{z})$ is a first integral for system \eqref{sys:H1}, hence $\hat{I}(\textbf{Z}) = \sum_{1\leq i \leq N} |Z_i|^2$ is a first integral for system \eqref{sys:H2}. Now that $\phi$ is orthogonal, we have $\sum_{1\leq i \leq N} |Z_i|^2 = \sum_{1\leq i \leq N} |W_i|^2$, while
$W_N = (q_N,p_N) = 0$, we see that actually $\sum_{1\leq i \leq N} |W_i|^2= \sum_{1\leq i \leq N-1} |W_i|^2$. In other words,
$\bar{I}(\textbf{W}) = \hat{I}(\textbf{Z})$.
\end{proof}
Recall we are interested in normalised orbits of the original Hamiltonian system \eqref{sys:H1}. According to results in the previous proposition, they can be characterized by the new coordinates, i.e.:
\begin{Pro}
\label{Pro:NewNormalOrbit}
\label{equivalence}
The orbits of system (\ref{sys:H3}) which satisfies $\bar{I}(\textbf{W}) = 1$ are the normalised orbits of the system \eqref{sys:H1}.
\end{Pro}
\subsection{Energy Surface in Reduced Phase Space} 
\label{subsec:3.2}
 The Hamiltonian system \eqref{sys:H3} with $\bar{H}(\textbf{W};W_N=0) : \mathbb{R}^{2N-2} \rightarrow \mathbb{R}$ is invariant under rotation, and $\bar{I}(\textbf{W})$ is the first integral.
By the theory of the standard symplectic reduction, we can fix $\bar{I} =1 $ and apply Hopf-fibration, it turns out that \eqref{sys:H3} canonically induces a Hamiltonian system 
\begin{equation}
\dot{\tilde{\textbf{W}}} = \textit{X}_{\tilde{H}}(\tilde{\textbf{W}} )  =\mathcal{\tilde{J}}(\tilde{\textbf{W}})\nabla \tilde{H}(\tilde{\textbf{W}}) \tag{H4} \label{sys:H4}
\end{equation}
 on $\mathbb{CP}^{N-2}$ \cite{abraham1978foundations}. Each point in $\mathbb{CP}^{N-2}$ represents a equivalent class of configurations up to the translation (by fixing $q_N= p_N=0$)  the rotation (by taking quotient of $\mathbb{SO}(2)$), and the homothety(by fixing $\bar{I}(\textbf{W})= 1$, thus $\nabla \bar{I}(\textbf{W}) \neq 0$). By Proposition \ref{equivalence}, each orbit on $\mathbb{CP}^{N-2}$ stands for a relative normalised orbit of system \eqref{sys:H1}. We resumed the whole reduction process in the following diagram:\\
\begin{tikzcd}
&&&&                     & \mathbb{S}^1 \arrow[hookrightarrow]{d} \\
&&& \mathbb{R}^{2N}  \arrow[hookleftarrow]{r}{q_N=p_N=0}&\mathbb{R}^{2N-2}  \arrow[hookleftarrow]{r}{\bar{I}= 1} &  \mathbb{S}^{2N-3}                \arrow [twoheadrightarrow]{d} {/ \mathbb{SO}(2)}     \\
&&&&             & \mathbb{CP}^{N-2}                     \\                                                  
\end{tikzcd}
\\
Although the energy surfaces for original Hamiltonian is not even bounded, due to the invariance under translation and the opposited singularities in logarithm function, the energy surface of the reduced Hamiltonian is indeed compact.
\begin{Lem}
\label{Lem:compact}
Let $c \in \mathbb{R}$. Consider the hypersurface $S_c = \tilde{H}^{-1}(c) \subset \mathbb{CP}^{N-2}   $. If $S_c \neq \emptyset$, then $S_c$ is compact. 
\end{Lem}
\begin{proof}:
Consider the set $\bar{S}_c=\bar{H}^{-1}(c)\cap \bar{I}^{-1}(1)$, which is the lifted set of $S_c$ from $\mathbb{CP}^{N-2}   $ to $\mathbb{S}^{2N-3}$.  If $\bar{S}_c$ is compact, then $S_c$ will be compact by quotient topology. First,  $\mathbb{S}^{2N-3}$ is a bounded manifold, hence the boundedness of $\bar{S}_c$. Next, recall that $\bar{I}(\textbf{W})= 1$ for all points in $\bar{S}_c$, which implies that all the mutual distances are bounded from above, since each squared mutual distance is a quadratic functions of $\textbf{W}$, as is shown in \eqref{eq:quatratic}. In other word, by the fact that $\bar{H}$ and $\bar{I}$ are preserved by the lifted flow of $\phi_{\bar{H}}$, the mutual distances cannot be too small. As a result, the energy surface $\bar{S}_c$ is isolated from singularity. But then the preimage of a closed set must be closed, hence $\bar{S}_c$  is closed. Hence $\bar{S}_c$ is compact. So is $S_c$.\\
\end{proof}
\subsection{Symplectic Capacity and Existence of Normalised Non-Trivial Relative Periodic Orbits}
\label{subsec:3.3}
We are now ready to prove the theorem concerning the existence of NTNRPOs of system \eqref{sys:H1}. Our main tool is the so called symplectic capacity, in particular the Hofer-Zehnder capacity $c_0$\cite{hofer2012symplectic}, which links periodic solution of Hamiltonian system to symplectic invariant. It is closely related to the searching of periodic orbits on a prescribed energy surface, initially studied by Rabinowitz \cite{rabinowitz1978periodic} and Weinstein \cite{weinstein1978periodic}. For general introduction to symplectic capacity theory one could turn to \cite{viterbo31capacites,hofer2012symplectic} and the references therein. 

\begin{Thm}
\label{Thm:ManyRPOLocal}
Suppose that $S_c = \tilde{H}^{-1}(c)$ is a non-empty regular hypersurface, then there exists a non-constant sequence $\lambda_k \rightarrow c$ and a sequence of normalised non-trivial relative periodic orbits  $\vz^k(t)$ of system \eqref{sys:H1} s.t. $H(\vz^k) = \lambda_k$.
\end{Thm}
\begin{proof}:
Since the hypersurface $S_c$ is regular, and by Lemma \ref{Lem:compact} it is compact. In other words, the vector field $\dot{\tilde{W}} = \frac{\nabla \tilde H(\tilde{W})}{|\nabla \tilde H(\tilde{W})|^2 }$ is locally well defined. By consequence we can almost surely extend $S_c$ to a 1-parameter family of regular energy surfaces $S(\delta)$, with $-\epsilon< \delta < \epsilon$ and $S(\delta) = S_{c+\delta}$.  
Define
\begin{align*} 
U_{\epsilon} = \bigcup_{\delta\in (-\epsilon, \epsilon)} S(\delta)
\end{align*} 
Let $c_0(\mathbb{CP}^{N-2},\omega)$ be the symplectic capacity, where $\omega= Im(g)$ and $g$ is the induced K{\"a}hler metric by the standard Hermitian, then $c_0(\mathbb{CP}^{N-2},\omega) = \pi <\infty$ (\cite[Corollary~1.5]{hofer1992weinstein}), thus \textit{a fortiori}, $c_0(U,\omega)< \infty$. Classical result of almost existence (\cite[Theorem~4.1]{hofer2012symplectic})  now implies the existence of infinitely many \textbf{non-constant} periodic solutions $\{\tilde{\textbf{W}}^k\}_{k\in \mathbb{N}}$ of the Hamiltonian system (\ref{sys:H4}) and a corresponding non-constant sequence $\{\lambda_k\}_{k\in \mathbb{N}}$, which satisfy that 
$\displaystyle \tilde{H}(\tilde{\textbf{W}}^k) = \lambda_k \rightarrow c$.\\
Now given a non-constant periodic orbit $\tilde{\textbf{W}}^k(t) = \phi_{\tilde{H}}(t) \subset \mathbb{CP}^{N-2}$  of system \eqref{sys:H4}, its lifted orbit $\textbf{z}^k =\phi_{H}(t) \subset \mathbb{R}^{2N}$ is a normalised relative periodic solution of the original Hamiltonian system \eqref{sys:H1}. We show that $\textbf{z}^k$ is not a relative equilibrium. 
Recall that by our construction of the reduced phase space, the vortex center of $\vz^k(t)$ is fixed at 0. If $\vz^k(t)$ is a relative equilibrium, then $\tilde{\textbf{W}}^k(t)$ is a fixed point in the reduced space, which contradicts the fact that $\tilde{\textbf{W}}^k(t)$ is a non-constant periodic solution. The theorem is thus proved.
\end{proof}
\begin{Rmk}
Strictly speaking the reduced dynamics is only defined on $\mathbb{CP}^{N-2} \setminus \tilde{\Delta}$.
Here $\tilde{\Delta}$ is projection of the generalized diagonal $\Delta$ where collision ( of two or multiple vortices) happens, i.e.,
\begin{align*}
\Delta = \{\vz \in \mathbb{R}^{2N} |\quad z_i = z_j \text{ for some } 1\leq i<j \leq N \}
\end{align*}
Fortunately, as we see in lemma \ref{Lem:compact} that the energy surface $\tilde{S}_c$ is bounded away from $\tilde{\Delta}$, this subtlety thus does not have impact on our proof.
\end{Rmk}

We have seen that the existence of infinitely many NTNRPOs depends on the existence of a regular energy surface of the reduced Hamiltonian. Since fixed points of the reduced Hamiltonian $\tilde{H}$ lift to normalised relative equilibria of the original Hamiltonian $H$. Thus to understand where are these NTNRPOs, we must have some information about the distribution of the set $\mathcal{H}_1$ in the set $\mathcal{H}_0$. But this has already been answered by theorem \ref{Thm:ClosedNull} and corollary \ref{Col:OpenDense}. We resume all the discussion above and theorem \ref{Thm:Main} is thus proved:
\begin{proof}[\textbf{of theorem \ref{Thm:Main}}]:
By combining theorem \ref{Thm:ManyRPOLocal} and corollary. Theorem \ref{Thm:ManyRPOLocal} implies that $\mathcal{H}_2$ is dense in $\mathcal{H}_0 \setminus \mathcal{H}_1 $. Corollary \ref{Col:OpenDense} implies that $\mathcal{H}_0 \setminus \mathcal{H}_1$ is dense in $\mathcal{H}_0$. As a result $\mathcal{H}_2$ is dense in $\mathcal{H}_0$.
\end{proof}

\begin{Rmk}
To know if there exists a periodic solution exactly on the prescribed energy surface, we need in general more condition, for example being of a contact type, see \cite{viterbo1987proof}.
\end{Rmk}

\section{Discrete Symmetric Reduction and Centre Symmetric Normalised Non-Trivial Relative Periodic Orbits}
\label{sec:4}
So far we have only considered the continuous symmetry, and have used the symplectic reduction to work in the reduced phase space. The factors that allowed us to find NTNRPOs are essentially:
\begin{enumerate}
\item The unitary change of variable;
\item Existence of regular and compact energy surface;
\item The finite symplectic capacity of the reduced spaces.
\end{enumerate}
On the other hand, one could alternatively impose discrete symmetric constraints on the orbits, which will largely reduce the degree of freedom until the reduced phase space is simple enough for explicit investigation. \\
The systematic investigation of this direction starts from Aref \cite{aref1982point}, where the double alternate ring configurations are studied in details. Then Koiller \cite{koiller1985aref} studied two and three vortex rings together with their bifurcations. One could turn to \cite{aref2002vortex} for the generalisation of previous results to various 2-dimensional manifolds. Later on, Tokieda, Souli{\`e}re, Montaldi and Laurent-Polz, among others, further generalized this method to find non-equilibrium (relative) periodic solutions of the so called "dansing vortices" on spheres and other manifolds under different symmetric group actions \cite{tokieda2001tourbillons,souliere2002periodic,montaldi2003vortex,laurent2004relative}. Essentially these existence results are based on two steps: In the first step, discrete symmetric reductions are carefully chosen to reduce the phase space to be 2-dimensional. Next, by fixing a regular energy level, one gets a 1-parameter curve in 1-dimensional compact space, which is diffeomorphic to a circle. As a result the (relative) periodic solutions are found.\\
In this chapter, we explain how to mix symplectic reduction and center symmetric reduction to get plenty of normalised non-trivial relative periodic solution with a center of symmetry. The whole idea is represented in the following example:
\begin{Emp}
Let $a_1, a_2, b_1,b_2$ be 4 vortices of positive vorticity. Moreover, the vorticities of $a_i$ and that of $b_i$ are the same, denoted by $\Gamma_i, i = 1,2$. Consider that at time 0, $a_i(0) = -b_i(0)$. Then by symmetry of the Hamiltonian, we see that $a_i(t)= -b_i(t),\forall t\in \mathbb{R}$. As a result, the Hamiltonian $H(a_1, a_2, b_1,b_2)$ could be considered as a system of 2 vortices:
\begin{align*}
H^{sym}(z)&= -\frac{1}{4\pi}(2\Gamma_{1}\Gamma_{2}(\log{|a_1 -a_2|^2} +\log{|a_1 +a_2|^2}) + \sum_{i=1}^{2}\Gamma_i^2 \log{|2a_i|^2})\\
\end{align*}
If we can find a relative periodic solution of this modified 2-vortex problem, we then will have actually found a symmetric relative periodic solution of the original 4-vortex problem. In particular, the above simplified Hamiltonian is still invariant under rotation. It turns out that, by mixing the discrete symmetry reduction with the symplectic reduction, the reduced phase space is \\
\begin{tikzcd}
&&&&                     & \mathbb{S}^1 \arrow[hookrightarrow]{d} \\
&&&\mathbb{R}^{8}\arrow[leftarrow]{r}{\vz=-\vz}&\mathbb{R}^{4}  \arrow[hookleftarrow]{r}{I= \frac{1}{2}}&  \mathbb{S}^{3}                \arrow [twoheadrightarrow]{d} {/ \mathbb{SO}(2)}     \\
&&&&             & \mathbb{CP}^{1}                     \\                                                  
\end{tikzcd}
\\
Now that each term in the logarithm is a quadratic function, and $I = 1$, we conclude that the non-vide energy hypersurfaces are compact. Moreover $a_1,a_2$ forms a relative equilibrium, if and only if $a_1, a_2, b_1,b_2$ also forms a symmetric relative equilibrium. 
\end{Emp} 
We claim the result more precisely:\\ 
\begin{Def}
Let $M,N\in \mathbb{N}$. We say a centred $M\times N$-vortex configuration is $C_N$-symmetric, if 
\begin{align}
\quad \vz = e^{\mathcal{J}_{M\times N}\frac{2\pi}{N}}\vz
\end{align}
We say a centred $M\times N$-vortex problem orbit is a $C_N$ symmetric orbit, if $\textbf{z}(t)$ is a $C_N$ symmetric configuration for all $t\in\mathbb{R}$.
\end{Def}
\begin{Emp}
Let $M = 3$ and $N = 4$, figure \ref{fig:4} shows roughly how these vortices are arranged at time 0.
\end{Emp}
\begin{Rmk}
A $C_N$ symmetric orbit is automatically a centred orbit.
\end{Rmk}
Now consider a $M\times N$-vortex problem, with M groups of vortices, and each group $M_l$ contains N vortices of the same vorticity $\Gamma_l>0$. At time 0, we put each group $M_l$ into a $C_N$ symmetric configuration, i.e., $\forall 1\leq i \leq N, 1\leq l\leq M$
\begin{align}
z_{li} &= e^{\mathbb{J}\frac{2\pi(i-1)}{N}}z_{l1} 
\end{align}
Then by symmetry of the Hamiltonian, we will have an orbit s.t. each vortices in each group $M_i, 1\leq i \leq M$ follow a $C_N$ symmetric orbit. We only need to study the Hamiltonian taking the $C_N$ symmetry into account. Denote $\displaystyle w_l = z_{l1},  1\leq l \leq M$ for short, which serves as a representative of the $N$ vortices in the $l$-th group $M_l$. We then consider the simplified Hamiltonian system 
\begin{equation}
\label{sys:Hsym} 
\Gamma \dot{\textbf{w}}(t) = \textit{X}_{H^{sym}}(\textbf{w}(t)) = \mathcal{J}_M\nabla H^{sym}(\textbf{w}(t)) \quad \textbf{w}=(w_1,w_2,...,w_M),\quad  w_i\in \mathbb{R}^2  \tag{H-Sym}
\end{equation}
where 
\begin{align*}
H^{sym}(\textbf{w})&= -\frac{1}{4\pi}\sum_{\substack{1\leq p,q\leq M\\ 1\leq i,j\leq N \\(p,i) \neq (q,j)}}\Gamma_p \Gamma_q \log{|e^{\mathbb{J}\frac{2\pi i}{N}}w_p-e^{\mathbb{J}\frac{2\pi j}{N}}w_q   |^2}   
\end{align*}
Clearly each periodic solution of the system \eqref{sys:Hsym} will imply a $C_N$ symmetric periodic solution of the original $M\times N$-vortex problem as in system \eqref{sys:H1}. If we further more require that $I(\textbf{w}) = \frac{1}{N}$, then it corresponds to a normalised $C_N$-symmetric periodic solution of the original $M\times N$-vortex problem as in system \eqref{sys:H1}.

\begin{figure}
\captionsetup{justification=centering}
\begin{center}
\includegraphics[width=40mm,scale=0.5]{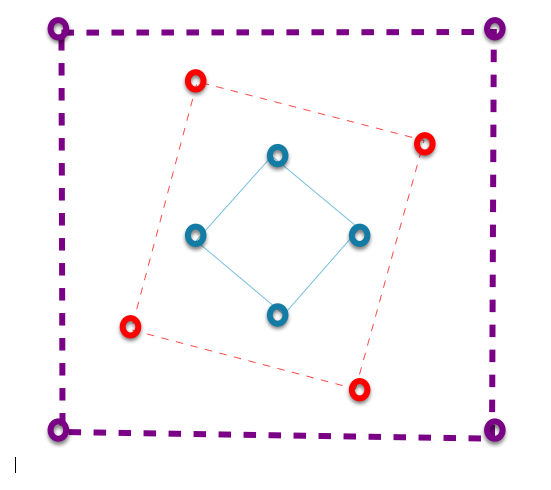}
\caption{An example of a $M\times N$-vortex configuration that is $C_N$ symmetric, with M=3, N=4}
\label{fig:4}
\end{center}
\end{figure}
We resume the above discussion in the following proposition:
\begin{Thm}
Consider the above symmetric $M\times N$-vortex problem with positive vorticities s.t. $\Gamma_{li} =\Gamma_{lj} ,1\leq l \leq M, 1\leq i<j \leq N$. Let 
\begin{align*}
&\mathcal{Z}^{sym}_{0}= \{\textbf{w}| \text{\textbf{w} is a } \text{ normalised orbit of the system } \eqref{sys:Hsym}  \}\\
&\mathcal{Z}^{sym}_2 = \{ \textbf{w}| \text{\textbf{w} is a } \text{ NTNRPO of the system }\eqref{sys:Hsym}\}\\
&\mathcal{H}^{sym}_{0} = \{h \in \mathbb{R} | h = H^{sym}(\textbf{w}), \textbf{z}\in \mathcal{Z}^{sym}_{0}\}\\
&\mathcal{H}^{sym}_{2} = \{h \in \mathbb{R} | h = H^{sym}(\textbf{w}), \textbf{z}\in \mathcal{Z}^{sym}_{2}\}
\end{align*}
Then $\mathcal{H}^{sym}_{2}$ is dense in $\mathcal{H}^{sym}_{0}$. In other words, there are infinitely many $C_N$-symmetric non-trivial normalised periodic solutions of the original $M\times N$-vortex problem in system \eqref{sys:H1}. 
\end{Thm}

\begin{proof}:
Similar as the discussion in theorem \ref{Thm:Main}.
\end{proof}

\begin{Rmk}
Again, since one doesn't need to worry about the degree of freedom, we can take M to be any positive integer, as long as there exists regular and compact energy surface in the (symplectically and symmetrically) reduced phase spaces.
\end{Rmk}
It is clear to see the advantages and drawbacks of our method compared to previous works:
\begin{enumerate}
\item The symplectic capacity does not have constraints on the dimension of reduced phase space, hence problem with more degrees of freedom could be considered.
\item The symplectic capacity does have constraints on the topology of reduced phase space, thus only applies to certain phase manifolds. In particular, in our planar case, we need a positive vorticity situation. 
\end{enumerate}
One can also compare our argument to variational methods with symmetry constraints applied in celestial mechanics, for example, the discussion in \cite{chenciner2000remarkable} \cite{chenciner2003action}. There, a variational argument, i.e., the Marchal's lemma (as a consequence of the minimisation of Lagrangian and a careful analysis of Kepler's orbit) is applied to get one part of the collision free orbit. Then the discrete symmetry is applied to complete the whole orbit. One can consider our argument as having a similar flavor, where the black box of Marchal's lemma is replaced by that of symplectic capacity. However, the symplectic capacity gives us the whole orbit directly. On the other hand, the admissible symmetry groups in our setting is less rich than those in the celestial mechanics setting, see for example \cite{chenciner2009unchained}.

\appendix
\section{Finiteness of Energy Levels For Normalised Equilibria With Positive Rational Vorticities}
\label{appendix:A}
In this appendix we prove theorem \ref{Thm:RationalFinite} in chapter \ref{sec:2}. First we give some definitions as preparation.
\begin{Def}
\label{Def:CAS}
A \textbf{closed algebraic set} is the locus of zeros of a collection of polynomials. 
\end{Def}
The following lemma is taken from Albouy and Kaloshin\cite{albouy2012finiteness}:
\begin{Lem}
\label{Lem:Dominating}
(\cite[page~540]{albouy2012finiteness})Let $X$ be a closed algebraic subset of $\mathbb{C}^N$ and $f:\mathbb{C}^N \rightarrow \mathbb{C}$ be a polynomial. Either the image $f(X) \subset \mathbb{C}$ is a finite set, or it is the complement of a finite set. In the second case one says that $f$ is \textbf{dominating}.
\end{Lem}
A necessary condition for a polynomial to be dominating is the following condition:
\begin{Lem}
\label{Lem:smoothpoint}
(\cite[page~42]{mumford1976algebraic}) A dominating polynomial $f$ on a closed algebraic subset possesses \textbf{smooth points}, i.e., points where the dimension of the tangent space is minimal and where $df \neq 0$.
\end{Lem}
Now back to our subject. Consider the Hamiltonian system
\begin{equation}
\label{sys:G1}
\Gamma \dot{\vz}(t) = \textit{X}_G(\vz(t)) = \mathcal{J}\nabla G(\vz(t)) \quad \dot{\vz}=(z_1,z_2,...,z_N),\quad  z_i\in \mathbb{R}^2 \tag{G1}
\end{equation}
\begin{align*}
G(z) = \prod_{1\leq i<j \leq N} |z_i-z_j|^{\Gamma_i\Gamma_j}   
\end{align*}
The relation between the Hamiltonian $G$ and Hamiltonian $H$ is justified by the relation $G(z) = exp\{-2\pi H(z)\}$. The dynamic interpretation of this reparametrization is that, in case of no collision, we reparametrise the orbit; while when ever collision happens, we replace the collision orbit by a fixed point. We define 
\begin{align*}
\mathcal{Z}_{1}(G) &=\{\textbf{z}\in R^{2N}| \text{\textbf{z} is a normalised relative equilibrium of the system } \eqref{sys:G1} \}\\
\mathcal{Z}^{2\pi}(G) &=\{\textbf{z}\in R^{2N}| \text{\textbf{z} is a relative equilibrium of the system } \eqref{sys:G1},\\ & \quad \quad \quad \quad  \text{ with minimal period T} =2\pi\}\\
\mathcal{G}_{1}& = \{g \in \mathbb{R} | g = G(\textbf{z}), \textbf{z}\in \mathcal{Z}_{1}(G)\}\\
\mathcal{G}^{2\pi}& = \{g \in \mathbb{R} | g = G(\textbf{z}), \textbf{z}\in \mathcal{Z}^{2\pi}(G) \}
\end{align*}
Note that for all relative equilibrium in $\mathcal{Z}^{2\pi}(G)$ the angular velocity $\displaystyle \omega =\frac{2\pi}{T}$ is fixed to be 1.
The first observation is the following rescaling property. Recall that $\displaystyle L = \sum_{1\leq i< j \leq N} \Gamma_i \Gamma_j$. 
\begin{Lem}
\label{Lem:Scaling}
Suppose $z(t)$ is an orbit of \eqref{sys:G1}. Then for $\lambda >0$, $\tilde{z}(t) = \lambda z(\lambda^{L-2}t )$ is also an orbit of \eqref{sys:G1}.
\end{Lem}
\begin{proof}:
This can be verified directly. Let $\tilde{z}(t) = \alpha z(\beta t)$. Since $z(t)$ is an orbit, we have
\begin{align} 
\dot{z}_i(t) = \mathbb{J} \nabla_{z_i}G(\textbf{z(t)})= \mathbb{J} \sum_{i\neq j}\Gamma_i\Gamma_j(\frac{G(\textbf{z}(t))}{|z_i(t)-z_j(t)|^{\Gamma_i\Gamma_j} } |z_i(t)-z_j(t)|^{\Gamma_i\Gamma_j-2} (z_i(t)-z_j(t))) 
\end{align}
As a result, we have 
\begin{align*}
\dot{\tilde{z}}_i(t) &= \mathbb{J}\alpha \beta  \nabla_{z_i} G(\textbf{z}(\beta t))\\
&= \mathbb{J}\alpha \beta\sum_{i\neq j}\Gamma_i\Gamma_j(\frac{G(\textbf{z}(\beta t))}{|z_i(\beta t)-z_j(\beta t)|^{\Gamma_i\Gamma_j} } |z_i(\beta t)-z_j(\beta t)|^{\Gamma_i\Gamma_j-2} (z_i(\beta t)-z_j(\beta t))) \\
&= \mathbb{J}\alpha^{2-L} \beta\sum_{i\neq j}\Gamma_i\Gamma_j(\frac{G(\alpha \textbf{z}(\beta t))}{|\alpha z_i(\beta t)- \alpha z_j(\beta t)|^{\Gamma_i\Gamma_j} } |\alpha z_i(\beta t)-\alpha z_j(\beta t)|^{\Gamma_i\Gamma_j-2} (\alpha z_i(\beta t)-\alpha z_j(\beta t)))\\
&= \mathbb{J} \alpha^{2-L} \beta \nabla_{z_i} G(\tilde{\textbf{z}}(t))
\end{align*}
Let $\alpha = \lambda, \beta =  \lambda^{L-2}$, the result follows.
\end{proof}
For a centred relative equilibrium  of \eqref{sys:G1}, the energy, the angular velocity and the angular momentum are closely related by the following lemma:
\begin{Lem}
Suppose now that \textbf{z} is a centred relative equilibrium  of \eqref{sys:G1}, with angular velocity $\omega$ and angular momentum $I(\textbf{z})$. Then 
\begin{enumerate}
\item $\displaystyle {\nabla G(\textbf{z}) = \frac{\omega}{2} \nabla I(\textbf{z}(t))}$
\item $\displaystyle \omega= \frac{LG}{I}   $
\end{enumerate}
\label{Lem:relationGI}
\end{Lem}
\begin{proof}:
1. This is direct consequence by the definition of the centred relative equilibrium.\\
2. Given that $\displaystyle {\nabla G(\textbf{z}) = \frac{\omega}{2} \nabla I(\textbf{z}(t))}$, we take inner product with $\textbf{z}$ on both sides and the result follows.
\end{proof}

\begin{Lem}
\label{Lem:IntegerFinite}
If $\Gamma_i \in \mathbb{N}^+$ and $\Gamma_i \geq 2$,  then $\mathcal{G}^{2\pi}$ is a finite set.
\end{Lem}
\begin{proof}:
Consider $\textbf{z}\in\mathcal{Z}^{2\pi}(G)$, it satisfies the following algebraic systems
\begin{align*}
\begin{pmatrix}
x_1\\
y_1
\end{pmatrix} &=\sum_{1\neq i} \Gamma_i \delta_{i1} \begin{pmatrix}
x_{1i}\\
y_{1i}
\end{pmatrix}\\
\begin{pmatrix}
x_2\\
y_2
\end{pmatrix} &=\sum_{2\neq i} \Gamma_i \delta_{i2} \begin{pmatrix}
x_{2i}\\
y_{2i}
\end{pmatrix}\\
&\mathrel{\makebox[\widthof{=}]{\vdots}}  \label{P} \tag{P}\\ 
\begin{pmatrix}
x_N\\
y_N
\end{pmatrix} &=\sum_{N\neq i} \Gamma_i \delta_{iN} \begin{pmatrix}
x_{Ni}\\
y_{Ni}
\end{pmatrix}\\
\end{align*}
where $x_{ij} = x_j - x_i$, $y_{ij} = y_j - y_i$, and $\displaystyle \delta_{ij} =
G(z) = (\prod_{\substack{1\leq p < q \leq N\\ (p,q)\neq (i,j)}} |z_p-z_q|^{\Gamma_p\Gamma_q}) |z_i-z_j|^{\Gamma_i\Gamma_j-2}$. If we consider $x_i$, $y_i$ and $\delta_{ij}$ as complex numbers, the system (\ref{P}) is a polynomial system in $\mathbb{C}^{2N}$. This system then defines a closed algebraic subset $\mathcal{A} \subset \mathbb{C}^{2N}$. \\
On the other hand, by lemma 2 in chapter 2, we see that $\displaystyle {\nabla G(\textbf{z}) = \frac{\omega}{2} \nabla I(\textbf{z}(t))}$ while $\displaystyle \omega= \frac{LG}{I}$. Taking $\omega = 1$, it turns out that for any $z\in \mathcal{G}^{2\pi}$, it satisfies 
\begin{align}
2\nabla G(\textbf{z}) =\nabla I(\textbf{z}(t)), \quad I = LG
\end{align}
Consider the function $g = 2G + I$ as a polynomial on $\mathcal{A}$. Since $dg = 0$ on $\mathcal{A}$,  $g$ does not possede any smooth point on $\mathcal{A}$. As a result $g$ is not a dominating polynomial due to lemma \ref{Lem:smoothpoint} . Thus according to lemma \ref{Lem:Dominating},   $g(\mathcal{A})$ contains only finitely many values in $\mathbb{C}$. But on $\mathcal{A}$, we must have $g= 2G + I = (L+2)G$. Since $L>0$ is a constant, we thus conclude that $G$ itself only gain finitely many values on $\mathcal{A}$. In otherwords, $\mathcal{G}^{2\pi}$ is a finite set.
\end{proof}

We have proved that relative equilibrium with fixed angular velocity only possedes finitely many energy levels. This however implies that relative equilibrium with fixed angular possedes only finitely many energy levels too. 
\begin{Lem}
If $\Gamma_i \in \mathbb{Q}^+$, then $\mathcal{G}_{1}$ is a finite set.  
\end{Lem}
\begin{proof}: First, we assume that $\Gamma_i \in \mathbb{N}^+$ and $\Gamma_i \geq 2$. In this case, 
Suppose to the contrary that $\{\textbf{z}^k\}_{k\in \mathbb{N}}\in \mathcal{Z}_1(G)$ s.t.
\begin{align}
0<G(\textbf{z}^1) < G(\textbf{z}^2) <G(\textbf{z}^3) ...<G(\textbf{z}^k) <... \label{increasing}
\end{align}
by lemma \ref{Lem:relationGI} their frequencies satisfy $\displaystyle \omega_k= \frac{LG(\textbf{z}^k)}{I(\textbf{z}^k)} = LG(\textbf{z}^k) >0$, moreover (\ref{increasing}) implies
\begin{align}
0<\omega_1 < \omega_2<...< \omega_k<... \label{omg} 
\end{align}
Now define $\displaystyle \tilde{\textbf{z}}^k(t) = (\omega_k)^{\frac{1}{2-L}} \textbf{z}^k(\frac{1}{\omega_k} t)$, by lemma \ref{Lem:Scaling}, $\displaystyle \tilde{\textbf{z}}^k \in \mathcal{G}^{2\pi}$. 
Then by lemma \ref{Lem:IntegerFinite}, $G(\tilde{\textbf{z}}^k)$ has only finite values. Again by lemma \ref{Lem:relationGI} ,  $I(\tilde{\textbf{z}}^k) = LG (\tilde{\textbf{z}}^k)$. Thus 
$\displaystyle I(\tilde{\textbf{z}}^k) = (\omega_k)^{\frac{2}{2-L}}$ has only finite values. By  (\ref{omg}) this leads to a contradiction. As a result, the lemma is proved.\\
Now for general case, suppose that $\displaystyle \Gamma_i = \frac{p_i}{q_i}\in \mathbb{Q}^+$. let $K = lcm(q_1,q_2,...,q_n)$ be the least common multiple of $q_1, q_2,..., q_N$. Consider now the new Hamiltonian 
\begin{equation}
\tilde{G} =  \prod_{1\leq i<j \leq N} |z_i-z_j|^{\tilde{\Gamma}_i\tilde{\Gamma}_j}  \label{sys:G2} \tag{G2}
\end{equation}
with $\displaystyle   \tilde{\Gamma}_i = 2K\Gamma_i, 1\leq i \leq N$. Now $ \tilde{\Gamma}_i \in \mathbb{N}^+$ and $\tilde{\Gamma}_i \geq 2$, thus we are back to previous situation. As a result $\tilde{\mathcal{G}}_{1}$ is a finite set. But note that $\displaystyle \tilde{G}(\textbf{z}) = (G(\textbf{z}))^{4K^2}$ and  $\displaystyle \mathcal{Z}_{1}(\tilde{G}) =  \mathcal{Z}_{1}(G)$, hence $\mathcal{G}_{1}$ itself is also a finite set and the lemma is proved.
\end{proof}
Now it is easy to prove Theorem \ref{Thm:RationalFinite}:
\begin{proof}
(\textbf{proof of Theorem \ref{Thm:RationalFinite}}):
Clearly $\mathcal{Z}_1(H) =  \mathcal{Z}_1(G)$ , and $\displaystyle G(\textbf{z}) = exp(-2\pi H(\textbf{z}) )$. Since $\mathcal{G}_{1}$ is a finite set, $\mathcal{H}_{1}$ is a finite set too.
\end{proof}
We have thus proved Theorem 1 under the assumption that $\Gamma_i \in \mathbb{Q}^+$. Some remarks might be useful:
\begin{Rmk}
Note that we have only proved the finiteness of energy surface for normalised relative equilibria, not the finiteness for normalised relative equilibria. 
\end{Rmk}
\begin{Rmk}
The switching from logarithm to polynomial serves to provide a linear relation between $G(\textbf{z})$ and $I(\textbf{z})$ when $\textbf{z}$ is a relative equilibrium. Actually, if we work directly with $H$, one verifies that 
$\displaystyle \nabla H(\textbf{z}) \textbf{z} = -\frac{L}{2\pi}$ for any orbit $\textbf{z}$, with is a constant and we cannot benifit from any homogeneous condition.  
\end{Rmk}

\section{Canonical Transformation In Symplectic Reduction: Five Vortices As An Example}
\label{CT}
In this appendix we use explicit canonical transformation to proceed the symplectic reduction for a 5-vortex problem. It can be generalized to $N$-vortex without any extra difficulty assuming that the vorticities are all positive. We will sometimes switch between real transformations and complex transformations without emphasizing it when no confusion should happen. To this end, consider the vortex Hamiltonian in the plane:
\begin{align*}
\Gamma \dot{\vz}(t) = X_H(\vz(t)) = \mathcal{J}\nabla H(\vz(t)) \quad \dot{\vz}=(z_1,z_2,...,z_N),\quad  z_i\in \mathbb{R}^{2}
\end{align*}
where
\begin{align*}
H(z)= -\frac{1}{4\pi}\sum_{i,j =1, i< j}^{N}\Gamma_{i}\Gamma_{j}\log{|z_i -z_j|^2}, \Gamma_i>0
\end{align*}
First, we make the change of variable 
$Z_i = (X_i, Y_i) = (\sqrt{\Gamma_i}x_i,\sqrt{\Gamma_i}y_i )$. It turns out that $\textbf{Z}=(Z_1,Z_2,...,Z_N)$ follows the usual Hamiltonian system 
\begin{align*}
\dot{\textbf{Z}}(t) = \textit{X}_H(\textbf{Z}(t)) = \mathcal{J}\nabla H(\textbf{Z}(t)) \quad \textbf{Z}=(Z_1,Z_2,...,Z_N),\quad  Z_i\in \mathbb{R}^{2}
\end{align*}
Now apply the following generalized Jacobi coordinates calculated in \cite{lim1989canonical}. Let 
$W = T Z$, where 
\begin{align*}
&A = \begin{bmatrix}
(\frac{\Gamma_1\Gamma_2}{\Gamma_1+\Gamma_2})^{\frac{1}{2}} &&&& \\
&(\frac{\Gamma_3\Gamma_4}{\Gamma_3+\Gamma_4})^{\frac{1}{2}} &&& \\
&&(\frac{(\Gamma_1+\Gamma_2)\Gamma_5}{\Gamma_1+\Gamma_2+\Gamma_5})^{\frac{1}{2}}&& \\
&&&(\frac{(\Gamma_1+\Gamma_2+\Gamma_5)(\Gamma_3+\Gamma_4)}{\Gamma_1+\Gamma_2+\Gamma_3+\Gamma_4+ \Gamma_5})^{\frac{1}{2}}  & \\
&&&& \frac{1}{\Gamma_1+\Gamma_2+\Gamma_3+\Gamma_4+ \Gamma_5}\\
\end{bmatrix}\\
&B=\begin{bmatrix}
-\frac{1}{\sqrt{\Gamma_1}}& \frac{1}{\sqrt{\Gamma_2}} & 0 & 0& 0 \\
0&0 &-\frac{1}{\sqrt{\Gamma_3}}& \frac{1}{\sqrt{\Gamma_4}}& 0\\
\frac{-\sqrt{\Gamma_1}}{\Gamma_1 + \Gamma_2}&\frac{-\sqrt{\Gamma_2}}{\Gamma_1 + \Gamma_2}&&&\frac{1}{\sqrt{\Gamma_5}}\\
\frac{-\sqrt{\Gamma_1}}{\Gamma_1 + \Gamma_2 + \Gamma_5}&\frac{-\sqrt{\Gamma_2}}{\Gamma_1 + \Gamma_2 + \Gamma_5}& \frac{\sqrt{\Gamma_3}}{\Gamma_3 + \Gamma_4}&\frac{\sqrt{\Gamma_4}}{\Gamma_3 + \Gamma_4}& \frac{-\sqrt{\Gamma_5}}{\Gamma_1 + \Gamma_2 + \Gamma_5}\\
\sqrt{\Gamma_1} &\sqrt{\Gamma_2} &\sqrt{\Gamma_3} &\sqrt{\Gamma_4} &\sqrt{\Gamma_5} \\
\end{bmatrix} 
\end{align*}
while $T = AB$. Then the transformation $T: \mathbb{C}^5 \rightarrow \mathbb{C}^5$
\begin{align*}
\textbf{W} = T \textbf{\textbf{Z}}, \quad (q,p) = T(X,Y) 
\end{align*}
when seen as a transformation $\mathbb{R}^{10} \rightarrow \mathbb{R}^{10}$, defines a symplectic transformation. \\
Now since $q_5$ is a first integral, $p_5$ is a cyclic variable. We can fix their value to be both 0 and reduce the Hamiltonian to 

\begin{align*}
\dot{\textbf{W}}(t) = X_H(\textbf{W}(t)) = \mathcal{J}\nabla H(\textbf{W}(t)) \quad \textbf{W}=(W_1,W_2,...,W_4 ; (0,0) ),\quad  W_i\in \mathbb{R}^{2}
\end{align*}
We see that the symplectic transformation T is linear, as a result the Hamiltonian expressed in the new conjugate variables $\textbf{W} $ is still invariant under rotation. In other words, 
\begin{align*}
\sum_{i=1}^{4} |W_i|^2 = \textbf{cst}
\end{align*}
is a first integral.
Next, consider the polar coordinates $W_i = (q_i,p_i) \rightarrow \bar{W}_i=(r_i,\theta_i)$ by letting
\begin{align*}
W_j = r_je^{i\theta_j}, \quad 0\leq r_j < \infty,    0\leq \theta_j < 2\pi, \quad 1\leq j\leq N-1 
\end{align*}
Let $\rho_j = \frac{r_j^2}{2}$. Consider the change of variable:
\begin{align*}
\bar{\textbf{W}}  = (\textbf{r}, \mathbf{\theta}) \rightarrow  
\tilde{\textbf{W}}= (\textbf{I}, \mathbf{\phi})
\end{align*}
where 
\begin{align*}
\begin{cases}
I_1 = \rho_1 + \rho_2 + \rho_3 + \rho_4 \\
I_2 = \rho_2 \\
I_3 = \rho_3 \\
I_4 = \rho_4
\end{cases}
\quad 
\begin{cases}
\phi_1 = \theta_1  \mod(2\pi)\\
\phi_2 = \theta_2 - \theta_1 \mod(2\pi)\\
\phi_3 = \theta_3 - \theta_1 \mod(2\pi)\\
\phi_4 = \theta_4 - \theta_1 \mod(2\pi)\\
\end{cases}
\end{align*}
It can be easily verified that $S:(\textbf{q},\textbf{p}) \rightarrow (\textbf{I}, \phi)$ is a symplectic transformation. Moreover, since $I_1$ is a first integral, $\phi_1$ is a cyclic variable. Thus we can fix $I_1 = 1$ and the resulted Hamiltonian represents the dynamics on $\mathbb{CP}^3$. \\

\section*{Acknowledgement}
The author thanks Dr. Jacques F{\'e}joz and Dr. Eric S{\'e}r{\'e} for many inspiring discussions.




\begin{thebibliography}{10}

\bibitem{abraham1978foundations}
R.~Abraham, J.~E. Marsden, and J.~E. Marsden.
\newblock {\em Foundations of mechanics}, volume~36.
\newblock Benjamin/Cummings Publishing Company Reading, Massachusetts, 1978.

\bibitem{albouy2012finiteness}
A.~Albouy and V.~Kaloshin.
\newblock Finiteness of central configurations of five bodies in the plane.
\newblock {\em Annals of mathematics}, pages 535--588, 2012.

\bibitem{aref1979motion}
H.~Aref.
\newblock Motion of three vortices.
\newblock {\em The Physics of Fluids}, 22(3):393--400, 1979.

\bibitem{aref1982point}
H.~Aref.
\newblock Point vortex motions with a center of symmetry.
\newblock {\em The Physics of Fluids}, 25(12):2183--2187, 1982.

\bibitem{aref2002vortex}
H.~Aref, P.~K. Newton, M.~A. Stremler, T.~Tokieda, and D.~L. Vainchtein.
\newblock Vortex crystals.
\newblock Technical report, Department of Theoretical and Applied Mechanics
  (UIUC), 2002.

\bibitem{bartsch2016periodic}
T.~Bartsch and Q.~Dai.
\newblock Periodic solutions of the n-vortex hamiltonian system in planar
  domains.
\newblock {\em Journal of Differential Equations}, 260(3):2275--2295, 2016.

\bibitem{bartsch2017global}
T.~Bartsch and B.~Gebhard.
\newblock Global continua of periodic solutions of singular first-order
  hamiltonian systems of n-vortex type.
\newblock {\em Mathematische Annalen}, 369(1-2):627--651, 2017.

\bibitem{borisov2004absolute}
A.~V. Borisov, I.~S. Mamaev, and A.~Kilin.
\newblock Absolute and relative choreographies in the problem of point vortices
  moving on a plane.
\newblock {\em Regular and Chaotic Dynamics}, 9(2):101--111, 2004.

\bibitem{borisov1998dynamics}
A.~V. Borisov and A.~Pavlov.
\newblock Dynamics and statics of vortices on a plane and a sphere—i.
\newblock {\em Regular and Chaotic Dynamics}, 3(1):28--38, 1998.

\bibitem{carminati2006fixed}
C.~Carminati, E.~Sere, and K.~Tanaka.
\newblock The fixed energy problem for a class of nonconvex singular
  hamiltonian systems.
\newblock {\em Journal of Differential Equations}, 230(1):362--377, 2006.

\bibitem{carvalho2014lyapunov}
A.~C. Carvalho and H.~E. Cabral.
\newblock Lyapunov orbits in the n-vortex problem.
\newblock {\em Regular and Chaotic Dynamics}, 19(3):348--362, 2014.

\bibitem{chenciner2003action}
A.~Chenciner.
\newblock Action minimizing solutions of the newtonian n-body problem: from
  homology to symmetry.
\newblock {\em arXiv preprint math/0304449}, 2003.

\bibitem{chenciner2009unchained}
A.~Chenciner and J.~F{\'e}joz.
\newblock Unchained polygons and the n-body problem.
\newblock {\em Regular and chaotic dynamics}, 14(1):64--115, 2009.

\bibitem{chenciner2000remarkable}
A.~Chenciner and R.~Montgomery.
\newblock A remarkable periodic solution of the three-body problem in the case
  of equal masses.
\newblock {\em Annals of Mathematics-Second Series}, 152(3):881--902, 2000.

\bibitem{ekeland2012convexity}
I.~Ekeland.
\newblock {\em Convexity methods in Hamiltonian mechanics}, volume~19.
\newblock Springer Science \& Business Media, 2012.

\bibitem{hampton2009finiteness}
M.~Hampton and R.~Moeckel.
\newblock Finiteness of stationary configurations of the four-vortex problem.
\newblock {\em Transactions of the American Mathematical Society},
  361(3):1317--1332, 2009.

\bibitem{hofer1992weinstein}
H.~Hofer and C.~Viterbo.
\newblock The {W}einstein conjecture in the presence of holomorphic spheres.
\newblock {\em Communications on pure and applied mathematics}, 45(5):583--622,
  1992.

\bibitem{hofer2012symplectic}
H.~Hofer and E.~Zehnder.
\newblock {\em Symplectic invariants and Hamiltonian dynamics}.
\newblock Birkh{\"a}user, 2012.

\bibitem{khanin1982quasi}
K.~Khanin.
\newblock Quasi-periodic motions of vortex systems.
\newblock {\em Physica D: Nonlinear Phenomena}, 4(2):261--269, 1982.

\bibitem{kirchhoff1876vorlesungen}
G.~R. Kirchhoff.
\newblock {\em Vorlesungen {\"u}ber mathematische physik: mechanik}, volume~1.
\newblock Teubner, 1876.

\bibitem{koiller1989non}
J.~Koiller and S.~P. Carvalho.
\newblock Non-integrability of the 4-vortex system: Analytical proof.
\newblock {\em Communications in mathematical physics}, 120(4):643--652, 1989.

\bibitem{koiller1985aref}
J.~Koiller, S.~P. De~Carvalho, R.~R. Da~Silva, and L.~C.~G. De~Oliveira.
\newblock On aref's vortex motions with a symmetry center.
\newblock {\em Physica D: Nonlinear Phenomena}, 16(1):27--61, 1985.

\bibitem{laurent2004relative}
F.~Laurent-Polz.
\newblock Relative periodic orbits in point vortex systems.
\newblock {\em Nonlinearity}, 17(6):1989, 2004.

\bibitem{lim1943motion}
C.~C. Lim.
\newblock {\em On the motion of vortices in two dimensions}.
\newblock Number~5. University of Toronto Press, 1943.

\bibitem{lim1989canonical}
C.~C. Lim.
\newblock Canonical transformations and graph theory.
\newblock {\em Physics Letters A}, 138(6-7):258--266, 1989.

\bibitem{lugt1983vortex}
H.~J. Lugt.
\newblock Vortex flow in nature and technology.
\newblock {\em New York, Wiley-Interscience, 1983, 305 p. Translation.}, 1983.

\bibitem{montaldi2003vortex}
J.~Montaldi, A.~Souliere, and T.~Tokieda.
\newblock Vortex dynamics on a cylinder.
\newblock {\em SIAM Journal on Applied Dynamical Systems}, 2(3):417--430, 2003.

\bibitem{mumford1976algebraic}
D.~Mumford.
\newblock Algebraic geometry. i, complex projective varieties.
\newblock 1976.

\bibitem{o2007relative}
K.~O’neil.
\newblock Relative equilibrium and collapse configurations of four point
  vortices.
\newblock {\em Regular and Chaotic Dynamics}, 12(2):117--126, 2007.

\bibitem{o1987stationary}
K.~A. O’Neil.
\newblock Stationary configurations of point vortices.
\newblock {\em Transactions of the American Mathematical Society},
  302(2):383--425, 1987.

\bibitem{palmore1982relative}
J.~I. Palmore.
\newblock Relative equilibria of vortices in two dimensions.
\newblock {\em Proceedings of the National Academy of Sciences},
  79(2):716--718, 1982.

\bibitem{poincare1892methodes}
H.~Poincar{\'e}.
\newblock {\em Les m{\'e}thodes nouvelles de la m{\'e}canique c{\'e}leste}
\newblock Gauthier-Villars, 1892.

\bibitem{rabinowitz1978periodic}
P.~H. Rabinowitz.
\newblock Periodic solutions of hamiltonian systems.
\newblock {\em Communications on Pure and Applied Mathematics}, 31(2):157--184,
  1978.

\bibitem{roberts2013stability}
G.~E. Roberts.
\newblock Stability of relative equilibria in the planar n-vortex problem.
\newblock {\em SIAM Journal on Applied Dynamical Systems}, 12(2):1114--1134,
  2013.

\bibitem{roberts2017morse}
G.~E. Roberts.
\newblock Morse theory and relative equilibria in the planar n-vortex problem.
\newblock {\em Archive for Rational Mechanics and Analysis}, pages 1--28, 2017.

\bibitem{routh1880some}
E.~J. Routh.
\newblock Some applications of conjugate functions.
\newblock {\em Proceedings of the London Mathematical Society}, 1(1):73--89,
  1880.

\bibitem{souliere2002periodic}
A.~Souli{\`e}re and T.~Tokieda.
\newblock Periodic motions of vortices on surfaces with symmetry.
\newblock {\em Journal of Fluid Mechanics}, 460:83--92, 2002.

\bibitem{synge1949motion}
J.~Synge.
\newblock On the motion of three vortices.
\newblock {\em Can. J. Math}, 1(3):257--270, 1949.

\bibitem{thomson1883treatise}
J.~J. Thomson.
\newblock {\em A Treatise on the Motion of Vortex Rings: an essay to which the
  Adams prize was adjudged in 1882, in the University of Cambridge}.
\newblock Macmillan, 1883.

\bibitem{tokieda2001tourbillons}
T.~Tokieda.
\newblock Tourbillons dansants.
\newblock {\em Comptes Rendus de l'Acad{\'e}mie des Sciences-Series
  I-Mathematics}, 333(10):943--946, 2001.

\bibitem{viterbo31capacites}
C.~Viterbo.
\newblock Capacit{\'e}s symplectiques et applications.
\newblock {\em S{\'e}minaire Bourbaki}, 31(1988-1989):1988--1989.

\bibitem{viterbo1987proof}
C.~Viterbo.
\newblock A proof of weinstein’s conjecture in {R}2n.
\newblock In {\em Annales de l'Institut Henri Poincare (C) Non Linear
  Analysis}, volume~4, pages 337--356. Elsevier, 1987.

\bibitem{helmholtz}
H.~von Helmholtz.
\newblock {Ü}ber integrale der hydrodynamischen gleichungen, welche den
  wirbelbewegungen entsprechen.
\newblock {\em Journal f{\"u}r Mathematik Bd. LV. Heft}, 1:4, 1858.

\bibitem{weinstein1978periodic}
A.~Weinstein.
\newblock Periodic orbits for convex hamiltonian systems.
\newblock {\em Annals of Mathematics}, 108(3):507--518, 1978.

\bibitem{yarmchuk1979observation}
E.~Yarmchuk, M.~Gordon, and R.~Packard.
\newblock Observation of stationary vortex arrays in rotating superfluid
  helium.
\newblock {\em Physical Review Letters}, 43(3):214, 1979.

\bibitem{ziglin1980nonintegrability}
S.~Ziglin.
\newblock Nonintegrability of a problem on the motion of four point vortices.
\newblock In {\em Sov. Math. Dokl}, volume~21, pages 296--299, 1980.

\end{thebibliography}


\end{document}